\documentclass[british]{amsart}

\pdfoutput=1

\usepackage{babel}
\usepackage[T1]{fontenc}
\usepackage[utf8]{inputenc}
\usepackage[varg]{pxfonts}

\usepackage{needspace}
\usepackage{setspace}
\usepackage{aliascnt}
\usepackage{varioref}
\labelformat{equation}{\textup{(#1)}}
\labelformat{enumi}{\textup{(#1)}}

\usepackage{amssymb,mathrsfs}

\usepackage{ifpdf}
\ifpdf
\usepackage{microtype}
\fi

\usepackage[pdfusetitle=true, colorlinks=true, linkcolor=blue, citecolor=blue, urlcolor=blue, linktoc=page]{hyperref}

\title{Homotopy BV algebras in Poisson geometry}

\author{C.~Braun}
\address{Centre for Mathematical Science\\
City University London\\
Northampton Square\\
London EC1V 0HB\\UK}
\email{Christopher.Braun.1@city.ac.uk}

\author{A.~Lazarev}
\address{Department of
Mathematics and Statistics\\
Lancaster University\\
Lancaster LA1 4YF\\UK}
\email{a.lazarev@lancaster.ac.uk}

\thanks{This work was partially supported by EPSRC grants EP/J00877X/1 and EP/J008451/1.}
\thanks{The authors would like to thank the Isaac Newton Institute for hospitality during this work.}
\subjclass[2010]{14D15, 16E45, 53D17}
\keywords{$L_\infty$~algebra, BV~algebra, Poisson manifold, differential operator}

\hypersetup{pdfauthor={C. Braun and A. Lazarev}}

\theoremstyle{plain}
\newtheorem{theorem}{Theorem}[section]
\newcommand{\newautoreftheorem}[2]{
\newaliascnt{#1}{theorem}\newtheorem{#1}[#1]{#2}\aliascntresetthe{#1}%
\expandafter\def\csname #1autorefname\endcsname{#2}}

\newautoreftheorem{proposition}{Proposition}
\newautoreftheorem{lemma}{Lemma}
\newautoreftheorem{corollary}{Corollary}

\theoremstyle{definition}
\newautoreftheorem{definition}{Definition}
\newautoreftheorem{remark}{Remark}
\newautoreftheorem{example}{Example}

\numberwithin{equation}{section} 
\numberwithin{figure}{section}   

\newcommand{\co}{\colon}
\newcommand{\ground}{\mathbf{k}}
\newcommand{\bvinf}{\mathrm{BV}_\infty}
\newcommand{\MC}{\mathrm{MC}}
\newcommand{\Der}{\mathrm{Der}}
\newcommand{\id}{\mathrm{id}}
\newcommand{\degree}[1]{\lvert #1 \rvert}
\DeclareMathOperator{\ad}{ad}
\DeclareMathOperator{\Ad}{Ad}
\DeclareMathOperator{\Hom}{Hom}
\DeclareMathOperator{\End}{End}
\DeclareMathOperator{\IHom}{\underline{Hom}}
\DeclareMathOperator{\CE}{CE}

\begin{document}
\def\sectionautorefname{Section}

\begin{abstract}
We define and study the degeneration property for $\bvinf$~algebras and show that it implies that the underlying $L_\infty$~algebras are homotopy abelian. The proof is based on a generalisation of the well-known identity $\Delta (e^{\xi})=e^{\xi}\left (\Delta(\xi)+\frac{1}{2}[\xi,\xi]\right )$ which holds in all BV algebras. As an application we show that the higher Koszul brackets on the cohomology of a manifold supplied with a generalised Poisson structure all vanish.
\end{abstract}

\maketitle
\tableofcontents

\section{Introduction}
A Batalin--Vilkovisky (BV) algebra is a graded commutative algebra supplied with an odd differential operator of order two and square zero. It appears in various contexts of algebraic topology, differential geometry and mathematical physics. As is usual for most algebraic structures, there is a notion of a \emph{homotopy}, or infinity, BV algebra structure encoding higher invariants of BV algebras. The general treatment of homotopy BV algebras is contained in~\cite{GalvezTonksVallette12:BV}, however for us the term $\bvinf$~algebra has a more restricted meaning, essentially equivalent to the definition in~\cite{kravchenko2000:deformationsofbvalgebras}.

An important special class of differential graded (dg) BV algebras is formed by imposing the \emph{degeneration property}, introduced in \cite{katzarkovkontsevichpantev2008:hodgetheoretic} and \cite{terilla2008:smoothnessbv}. This property holds in e.g.~the de Rham algebra of a symplectic manifold or the Dolbeault algebra of a Calabi--Yau manifold and in favourable cases it leads to a construction of a formal Frobenius manifold~\cite{BarannikovKontsevich98:Frobenius,Merkulov98:Frobenius}.

A dg BV algebra supports the structure of a dg Lie algebra, whereas a $\bvinf$~algebra gives rise to a homotopy analogue of a Lie algebra, called an $L_\infty$~algebra. The ordinary degeneration property for a dg BV algebra implies that its underlying dg Lie algebra is \emph{homotopy abelian}, i.e.~that it is quasi-isomorphic to an abelian Lie algebra. We will prove a generalisation of this statement for $\bvinf$~algebras. The proof is based on a generalisation of the following well-known identity for ordinary dg BV algebras:
\[
\Delta(e^\xi) = e^\xi\left ( \Delta(\xi) + \frac{1}{2}[\xi,\xi] \right )
\]
This generalisation holds, essentially, for all operators $\Delta$, not necessarily of second order.

Our main application of the degeneration property for $\bvinf$~algebras concerns the structure of the de Rham algebra of a manifold $M$ supplied with a generalised Poisson structure. The latter is just a multivector field on $M$ whose Schouten bracket with itself is zero. An ordinary Poisson structure (a bivector field) on $M$ gives rise to a Koszul bracket on $\Omega(M)$, the de Rham algebra of $M$, making the latter a dg Lie algebra (in fact, a dg BV algebra). It was shown in \cite{voronovkhudaverdian2008:higerpoissonbrackets} that a generalised Poisson structure on $M$ leads to an $L_\infty$~structure on $\Omega(M)$ (in fact, to a $\bvinf$~structure on $\Omega(M)$). We show that this $\bvinf$~algebra has the degeneration property. As a consequence the higher Koszul brackets on the cohomology of $M$ vanish.

\subsection{Notation and conventions}
Throughout this paper $\ground$ will denote a field of characteristic zero.

We choose to work in the supergraded context. More precisely, this means we work in the category of super vector spaces: $\mathbb{Z}/2\mathbb{Z}$--graded $\ground$--linear vector spaces, with morphisms linear maps preserving the grading. This is a symmetric monoidal category with symmetry isomorphism $s\co V\otimes W \to W\otimes V$ given by
\[
s(v\otimes w) = (-1)^{\degree{v}\degree{w}}w\otimes v.
\]
Denote by $\Pi\ground$ the one dimensional super vector space concentrated in odd degree. We denote by $\Pi V$ the functor $V\mapsto \Pi\ground \otimes V$, called \emph{parity reversion}. The space $\IHom(V,W)$ denotes the super vector space with even part the space of morphisms $V\to W$ (the linear maps which preserve the grading) and odd part the space of morphisms $V\to \Pi W$ (the linear maps which reverse the grading). This is an internal $\Hom$ functor making the category of super vector spaces into a closed symmetric monoidal category.

For brevity we will normally suppress the adjective `super'. In particular by a (unital) associative/commutative\footnote{Commutative algebras are always assumed to be associative.}/Lie algebra we will always mean the appropriate notions in this category. This means that commutativity, anti-commutativity and the Jacobi identity are to be understood in the graded sense, for example commutativity would mean
\[
ab = (-1)^{\degree{a}\degree{b}}ba.
\]
The abbreviation `dg' will stand for `differential (super)graded' and we will abbreviate the expression `commutative dg algebra' to `cdga'.

We note that most of our results also hold in the $\mathbb{Z}$--graded context, after making suitable cosmetic adaptations such as replacing parity reversion with suspension/desuspension as appropriate.

We will use the notion of a \emph{complete} (dg) vector space; this is just an inverse limit of finite-dimensional (dg) vector spaces. An example of a complete vector space is $V^*$, the $\ground$--linear dual to a discrete vector space~$V$. A complete vector space comes equipped with a topology and whenever we deal with a complete vector space all linear maps from or into it will be assumed to be continuous; thus we will always have $V^{**}\cong
V$. Similarly, we will always have $(V\otimes V)^*\cong V^*\otimes V^*$ since tensor products of complete vector spaces $C = \lim_\leftarrow C_i$ and $B=\lim_\leftarrow B_i$ will always be assumed to mean \emph{completed} tensor products, in other words $C\otimes B$ is the complete vector space $C\otimes B = \lim_\leftarrow C_i \otimes B_j$. If $V$ is a discrete vector space and $C = \lim_{\leftarrow} C_i$ is a complete vector space, the tensor product $C\otimes V$ will always be assumed to mean $\lim_\leftarrow C_i\otimes V$.

A \emph{complete algebra} is an algebra in complete spaces which, in addition, is also local. A prototypical example of a complete algebra is the completed symmetric algebra $\widehat{S}V$ on a complete vector space~$V$.

\section{\texorpdfstring{$L_\infty$}{L-infinity}~algebras}
This introductory section fixes the terminology and standard facts about $L_\infty$~algebras relevant to the present work. More detailed discussion can be found in, e.g.~\cite{ChuangLazarev:twisting}.

Let $V$ be a vector space; then its dual is a complete vector space and we can form its complete symmetric algebra $\widehat{S}V^*$. Let us denote by $\Der (\widehat{S}V^*)$ the Lie algebra of (continuous) derivations of $\widehat{S}V^*$. Choosing a basis $x_i,i\in I$ in $V^*$, any derivation $\xi\in\Der(\widehat{S}V^*)$ can be written as
 \[\xi=\sum_{i\in I}f_i^0\partial_{x_i}+\sum_{i\in I} f_i^1\partial_{x_i}+\dots+\sum_{i\in I} f_i^k\partial_{x_i}+\dots\]
where $f_i^k$ is a linear combination (perhaps infinite if the indexing set $I$ is infinite) of monomials in $x_n$s of order $k$. If $\xi=\sum_{i\in I} f_i^n\partial_{x_i}$ for a fixed $n$ then we say that $\xi$ is a derivation of order $n$; this notion clearly does not depend on the choice of the basis in $V^*$. The space of derivations of order $\geq n$ will be denoted by $\Der_{\geq n}(\widehat{S}V^*)$.

\begin{definition}
Let $V$ be a vector space. An $L_\infty$~structure on $V$ is an odd element $m\nobreak\in\nobreak\Der_{\geq 1}(\widehat{S}\Pi V^*)$ which satisfies the equation $m^2=m\circ m\nobreak=\nobreak 0$. The pair $(V,m)$ will be referred to as an $L_\infty$~algebra and the algebra $\widehat{S}\Pi V^*$, supplied with the differential $m$, as its representing complete cdga.
\end{definition}

There is a concomitant notion of an $L_\infty$~map.

\begin{definition}
Let $(V,m_V)$ and $(W,m_W)$ be two $L_\infty$~structures on $V$ and $W$. An $L_\infty$~map $f\co(V,m_V)\to(W,m_W)$ is a map between their representing complete cdgas $f\co\widehat{S}\Pi W^*\to \widehat{S}\Pi V^*$ such that $f\circ m_W=m_V\circ f$.
\end{definition}

A more traditional approach to defining $L_\infty$~algebras and maps is through multilinear maps. Note that a derivation $m\in\Der_{\geq 1}(\widehat{S}\Pi V^*)$ has the form $m=m_1^*+m_2^*+\dots$ where $m_n^*$ is a derivation of order $n$. In other words, any derivation is determined by the collection of maps $m_n^* \co \Pi V^*\to\left ((\Pi V^*)^{\otimes n}\right )_{S_n}$. We have an identification between $S_n$ coinvariants and $S_n$ invariants:
\[
i_n\co \left((\Pi V^*)^{\otimes n}\right)_{S_n}\to
\left((\Pi V^*)^{\otimes n}\right)^{S_n}\cong
\left((\Pi V^{\otimes n})_{S_n}\right)^*
\]
where $i_n(x_1\otimes \dots \otimes x_n)=\sum_{\sigma\in S_n}\sigma[x_1\otimes\dots\otimes x_n]$. Then the dual to the composite map
\[
i_n\circ m_n^*\co \Pi V^*\to \left((\Pi V^{\otimes n})_{S_n}\right)^*
\]
is a map $m_n \co(\Pi V^{\otimes n})_{S_n}\to\Pi V$. Thus an $L_\infty$~structure on $V$ is equivalent to a collection of symmetric multilinear maps $m_n\co (\Pi V)^{\otimes n}\to \Pi V$ of odd degree as above subject to appropriate conditions stemming from the equation $m\circ m=0$. For example, the linear component $m_1\co \Pi V \to \Pi V$ is a differential on $\Pi V$.

A similar argument shows that an $L_\infty$~map $f\co \widehat{S}\Pi W^*\to \widehat{S}\Pi V^*$ is equivalent to a collection of symmetric multilinear maps $f_n\co (\Pi V)^{\otimes n}\to \Pi W$ of even degree satisfying suitable identities.

We can now define the notion of an $L_\infty$~(quasi-)isomorphism.

\begin{definition}
 An $L_\infty$~map $f\co (V,m_V)\to (W,m_W)$ is an $L_\infty$~(quasi-)isomorphism if its linear component $f_1\co \Pi V \to \Pi W $ is a (quasi-)isomorphism where $\Pi V$ and $\Pi W$ are supplied with differentials the linear components of $m_V$ and $m_W$.
\end{definition}

\subsection{Minimal \texorpdfstring{$L_\infty$}{L-infinity}~algebras and homotopy abelian \texorpdfstring{$L_\infty$}{L-infinity}~algebras}
An important special class of $L_\infty$~algebras is formed by \emph{minimal} $L_\infty$~algebras; these are derivations $m\in\Der_{\geq 2}(\widehat{S}\Pi V^*)$. Minimal $L_\infty$~algebras have the striking property that any $L_\infty$~quasi-isomorphism between them must be an isomorphism. Any $L_\infty$~algebra is $L_\infty$~quasi-isomorphic to a minimal one (cf.~for example \cite[Lemma 4.9]{kontsevich2003:defquant}), called its \emph{minimal model}. Any two minimal models of a given $L_\infty$~algebra are (non-canonically) isomorphic.

An $L_\infty$~algebra is \emph{abelian} if all its higher brackets $m_n$ vanish for $n\geq 2$. It is \emph{homotopy abelian} if it is $L_\infty$~quasi-isomorphic to an abelian $L_\infty$~algebra. Clearly an $L_\infty$~algebra is homotopy abelian if and only if its minimal model $(V,m_V)$ has its $L_\infty$~structure $m_V$ identically vanishing: $m_V=0$.

\subsection{Maurer--Cartan elements in \texorpdfstring{$L_\infty$}{L-infinity}~algebras}
\Needspace*{4\baselineskip}

\begin{definition}
Let $m$ be an $L_\infty$~structure on a vector space $V$ and let $C$ be a complete cdga with maximal ideal $C_+$. Then an even element $\xi\in C_+\otimes \Pi V$ is \emph{Maurer--Cartan} if it satisfies the Maurer--Cartan equation
 \[ (d_C\otimes \id)(\xi) +
\sum_{i=1}^\infty
\frac{1}{i!}m^C_i(\xi,\dots,\xi) = 0\]
where $m_i^C$ is obtained by extending $m_i$ multilinearly in $C$. The set of Maurer--Cartan elements in $C_+\otimes \Pi V$ will be denoted by $\MC(V,C)$. The correspondence $(V,C)\mapsto \MC(V,C)$ is functorial in $C$.
\end{definition}

The significance of Maurer--Cartan elements stems from the following standard result.

\begin{proposition}\label{MCfunctor}
Let $(V,m)$ be an $L_\infty$~algebra and $C$ be a complete cdga. Then the functor $C\mapsto\MC(V,C)$ is represented by the complete cdga $(\widehat{S}\Pi V^*,m)$.
\end{proposition}

\begin{remark}\label{rem:Yoneda}
An $L_\infty$~map $V\to W$ gives rise, for any complete cdga $C$, to a map of sets $\MC(V,C)\to\MC(W,C)$ which is functorial in $C$, in other words a natural transformation. In fact such a natural transformation is, by Yoneda's lemma, equivalent to having an $L_\infty$~map $V\to W$. This observation is often useful for constructing $L_\infty$~maps.
\end{remark}

\section{\texorpdfstring{$\bvinf$}{BV-infinity}~algebras and differential operators}
For any associative algebra $A$, by a linear operator on $A$ we mean an element of the associative algebra $\End(A) = \IHom(A,A)$. Any element of $a$ is regarded as a linear operator on $A$ by $a(x) = ax$.

Let $A$ be a unital associative algebra with pronilpotent ideal $I$. Then for any $a\in I$ we define:
\[
e^{a} = 1 + a + \frac{a^2}{2!} + \frac{a^3}{3!} + \dots \in A
\]
Given any $b\in A$ such that $b - 1 \in I$ we define:
\[
\log{b} = (b-1) - \frac{(b-1)^2}{2} + \frac{(b-1)^3}{3} - \dots \in I
\]
For any $a\in I$ we have $\log{e^a} = a$ and for any $b\in A$ with $b-1\in I$ we have $e^{\log{b}} = b$.

Recall the following well-known identity, which holds for any $a\in I$ and any $D$ a continuous linear operator on $A$:
\[
\Ad(e^a)(D) = e^a D e^{-a} = e^{\ad(a)}(D)
\]
Rearranging, we obtain the following version of this identity, which will be important for us.
\begin{equation}\label{eq:exponentialidentity}
D e^{a} = e^{a}e^{-\ad(a)}(D)
\end{equation}

\subsection{Differential operators}
Recall that if $A$ is a unital commutative algebra then there is an increasing filtration $F_0\subset F_1 \subset \dots \subset \End(A)$ of $\End(A)$ where $F_n$ is the space of differential operators of order not higher than $n$.

More precisely, set $F_{-1} = \{0\}$ and define recursively
\[F_n = \{D\in \End(A) \co \forall a\in A, [D,a] \in F_{n-1}\}.\]
Unwrapping this recursive definition, $D$ is a differential operator of order not higher than $n$ if for any $a_1,\dots, a_{n+1}\in A$ it holds that
\[
[[[\dots[D,a_1]\dots],a_n],a_{n+1}] = 0.
\]

Given $D\in \End(A)$, for $n\geq 0$ define the linear maps $m_n\co A^{\otimes n} \to A$ by
\[
m_n(a_1,\dots, a_n) = [[[\dots[D,a_1]\dots],a_{n-1}],a_n](1).
\]
Then $D$ is a differential operator of order not higher than $n$ if and only if $m_{n+1} = 0$ if and only if $m_{n+i} = 0$ for all $i\geq 1$.

\begin{remark}
The terminology `$n$--th order differential operator' is a more elegant phrase, but technically speaking it could be ambiguous. We will use it to mean a differential operator of order not higher than $n$.
\end{remark}

\subsection{\texorpdfstring{$L_\infty$}{L-infinity}~algebras from linear operators}
Since $A$ is commutative the maps $m_n$ are symmetric. Given an operator $D$ then we obtain maps $m_n^*\co A^* \to \left ((A^*)^{\otimes n}\right )_{S_n}$ and hence a derivation $m\in\Der(\widehat{S}A^*)$. We have the following observation due originally to Kravchenko.

\begin{proposition}[{\cite[Propostion 2]{kravchenko2000:deformationsofbvalgebras}}]\label{prop:krav}
$D^2=0$ if and only if $m^2=0$.
\end{proposition}

We now have the following corollary, generalising in a certain way the construction of the Lie bracket associated to a BV operator. This can also be realised as an instance of Voronov's construction of higher derived brackets \cite{voronov2005:higherderivedbrackets,voronov2005:higherderivedbracketsarbitrary}.

\begin{corollary}\label{cor:linftyconstruction}
Let $A$ be a unital commutative algebra and let $D$ be a linear operator of odd degree such that $D^2=0$ and $D(1)=0$. Then $m\in\Der(\widehat{S}A^*)$ defines an $L_\infty$~structure on $\Pi A$.\qed
\end{corollary}

\begin{remark}
If $D(1)\neq 0$ then $m$ will in fact define a \emph{curved} $L_\infty$~algebra with curvature $D(1)$. Many of the results we state will also hold after adding the adjective `curved' at appropriate places. However, since we will not make use of curved $L_\infty$~algebras we will not consider this, for the sake of simplicity.
\end{remark}

This $L_\infty$~structure is always homotopy abelian, which we shall now show. We first prove the following straightforward lemma.

\begin{lemma}\label{lem:expmc}
Let $A$ be a unital commutative algebra with a linear operator of odd degree with $D^2 =0$ and $D(1)=0$ giving rise to an $L_\infty$~structure on $\Pi A$ and let $C$ be a complete cdga with maximal ideal $C_+$. Then an even element $\xi \in C_+ \otimes A$ satisfies the Maurer--Cartan equation if and only if
\[
e^{-\ad(\xi)}(d_C + D)(1) = 0.
\]
\end{lemma}

\begin{proof}
Note that since $\xi \in C_+ \otimes A$ the left hand side converges. The operations $m_n^C\co (C \otimes A)^{\otimes n}\to C\otimes A$ satisfy:
\[
m_n^C(\xi,\dots,\xi) = (-\ad(\xi))^n(D)(1)
\]
Therefore, unwrapping this formula explicitly and remembering that $d_C$ is in particular a first order differential operator we obtain:
\[
e^{-\ad(\xi)}(d_C + D)(1) = \sum_{i=0} \frac{1}{i!}(-\ad(\xi))^i (d_C+D)(1) = (d_C\otimes \id)(\xi) + \sum_{i=1} \frac{1}{i!}m_i^C(\xi,\dots,\xi)
\]
In other words, this equation is simply the Maurer--Cartan equation.
\end{proof}

\begin{remark}
The proof of \autoref{lem:expmc} taken together with \autoref{eq:exponentialidentity} in fact generalises to the $L_\infty$~context the standard identity for a BV algebra with second order BV operator $D = \Delta$:
\[
\Delta (e^\xi) = e^\xi\left(\Delta(\xi) + \frac{1}{2}[\xi,\xi]\right)
\]
Indeed, unwrapping \autoref{eq:exponentialidentity} in terms of the higher $L_\infty$~operations we obtain (under appropriate continuity and convergence conditions) for any operator $D$ with $D(1)=0$:
\[
D(e^\xi) = e^\xi \left( D(\xi) + \frac{1}{2!}m_2(\xi,\xi) + \frac{1}{3!}m_3(\xi,\xi,\xi) + \dots \right )
\]
\end{remark}

\begin{theorem}\label{thm:homotopyabelian}
Let $A$ be a unital commutative algebra with a linear operator of odd degree with $D^2 =0$ and $D(1)=0$ giving rise to an $L_\infty$~structure on $\Pi A$. Then this $L_\infty$~algebra is homotopy abelian.
\end{theorem}

\begin{proof}
Let $C$ be a complete cdga with maximal ideal $C_+$. Since $d_C(1) = D(1)=0$, together with \autoref{eq:exponentialidentity} we have for $\xi \in C_+ \otimes A$:
\[
(d_C + D) (e^{\xi}-1) = e^{\xi}\left (e^{-\ad(\xi)}(d_C + D)(1)\right )
\]
Therefore, for $\xi$ of even degree, by \autoref{lem:expmc} $e^\xi - 1$ is a $(d_C + D)$--cycle if and only $\xi$ is Maurer--Cartan. So the invertible map $\xi \mapsto e^\xi - 1$ gives rise to a natural isomorphism of Maurer--Cartan sets
\[\MC(\Pi A, C)\cong \MC(\mathfrak{h},C)\]
where $\mathfrak{h}$ is the abelian $L_\infty$~algebra with underlying space $\Pi A$, differential $m_1 = D$ and $m_n=0$ for all $n\geq 2$. It follows from \autoref{MCfunctor} that $\Pi A$ and $\mathfrak{h}$ are $L_\infty$~isomorphic.
\end{proof}

\begin{remark}
The proof of \autoref{thm:homotopyabelian} simply constructs an automorphism of commutative algebras $\widehat{S}A^* \to \widehat{S}A^*$, commuting with two differential operators on $\widehat{S}A^*$. In particular, this construction does not depend in an essential way on these differential operators being of square zero, however if one of them squares to zero then so must the other. Since one operator is the derivation $m$ associated to $D$ and the other is just $D^*$, we could obtain in this way a simple, conceptual and non-computational proof of \autoref{prop:krav}. Indeed, the $L_\infty$~algebra obtained in this way is merely a gauge transformation of the abelian $L_\infty$~algebra with differential $D$.
\end{remark}

\subsection{Commutative \texorpdfstring{$\bvinf$}{BV-infinity}~algebras}
First we recall the definition of a dg BV algebra.

\begin{definition}
A dg BV algebra is a unital commutative algebra $A$ with odd operators $d,\Delta$ with the following properties:
\begin{itemize}
\item $d^2 = \Delta^2 = d\Delta + \Delta d = 0$
\item $d(1) = \Delta(1) = 0$
\item The operator $d$ is a first order differential operator and the operator $\Delta$ is a second order differential operator.
\end{itemize}
\end{definition}

\begin{remark}
Note that the first two conditions are equivalent to saying that $D = d + h\Delta\in \End(A)[[h]]$ satisfies $D^2 = 0 $ and $D(1)=0$.
\end{remark}

We make the following definition, which is a slight modification of the original definition proposed by \cite{kravchenko2000:deformationsofbvalgebras}.

\begin{definition}
A \emph{(commutative) $\bvinf$~algebra} is a unital commutative algebra $A$ with odd operators $D_0, D_1, D_2,\dots$ with the following properties:
\begin{itemize}
\item $D = D_0 + hD_1 + h^2D_2 + \dots\in \End(A)[[h]]$ satisfies $D^2=0$ and $D(1)=0$.
\item Each $D_i$ is an $(i+1)$--th order differential operator.
\end{itemize}
\end{definition}

\begin{example}
One of the simplest examples of a BV algebra is the Chevalley--Eilenberg \emph{homological} complex of a Lie algebra. Let $V$ be a Lie algebra and denote by $\CE_\bullet(V)$ the complex for which $\CE_n(V)=\Lambda^n(V)$ and the differential $\Delta$ is given by the standard formula, cf.~for example \cite[Section 10.1]{loday1992:cyclichombook}. Then $\Delta$ is not a derivation of the wedge product on $\CE_\bullet(V)$, but rather an operator of order two. Furthermore, if $V$ itself has a differential then $\CE_\bullet(V)$ becomes a dg BV algebra.

Now let $V$ be an $L_\infty$~algebra and $(\widehat{S}\Pi V^*,m)$ be its representing complete cdga. In this case $\CE_\bullet(V)$ is the symmetric (but not completed) algebra $S\Pi V$ and the differential $\Delta$ is simply the dual to $m$ under the duality isomorphism $(\widehat{S}\Pi V^*)^*\cong S\Pi V$. Note that if $m=\sum_{i=1}^\infty m_i$ then $\Delta=\sum_{i=1}^\infty\Delta_i$ where $\Delta_i$ is the operator dual to $m_i$. Note that the operator $m_i$ is a sum of operators which are compositions of a derivation of order zero followed by an operator of multiplication with a monomial of order $i$. It follows that its dual $\Delta_i$ is a differential operator of order $i$ and since $\Delta^2=0$ we conclude that $(\CE_\bullet(V),\Delta_1, \Delta_2, \dots)$  is a $\bvinf$~algebra.
\end{example}

The following generalisation of the degeneration property for dg BV algebras will be a crucial property for us.

\begin{definition}
A $\bvinf$~algebra is said to have the \emph{degeneration property} if for every $N$ the homology of $A[h]/(h^N)$ with respect to the differential induced by $D$ is a free $\ground[h]/(h^N)$--module.
\end{definition}

\begin{remark}
The degeneration property clearly depends only on the properties of the operator $D = D_0 + hD_1 + h^2D_2 + \dots\in \End(A)[[h]]$, but not on the commutative algebra structure of $A$. It is easy to see that the degeneration property is equivalent  to the collapsing at the $E_1$ term of the spectral sequence associated with the filtration of $A[[h]]$ by the powers of $h$. It is further equivalent to the existence of so-called Hodge to de Rham degeneration data on the multicomplex $(A,D_0,D_1,\dots)$, see \cite{dotsenkoshadrinvallette:derhamhomotopyfrob} concerning this terminology. Yet another way to formulate the degeneration property is to require that the differential $D$ is a \emph{trivial} formal deformation of the differential $D_0$.
\end{remark}

Given a $\bvinf$~algebra $A$, then by \autoref{cor:linftyconstruction} there is an associated $L_\infty$~structure $ m_1 = D , m_2 , m_3 , \dots$ on $\Pi A[[h]]$ arising from the operator $D$. By \autoref{thm:homotopyabelian} this $L_\infty$~algebra is homotopy abelian.

However, since $m_n|_{h=0}=0$ for all $n\geq 2$, this $L_\infty$ structure is not quite the correct generalisation of the Lie algebra associated to a dg BV algebra. We instead wish to consider the $L_\infty$~structure $D, m_2/h, m_3/h^2, \dots$. At first glance this may not appear to define an $L_\infty$~structure on $\Pi A[[h]]$, however it turns out to do so.

\begin{proposition}\label{prop:linftyrescale}
Let $A$ be a $\bvinf$~algebra and let $m_1, m_2, m_3, \dots$ be the associated $L_\infty$~structure on $\Pi A[[h]]$. Then the sequence of maps $m_1, m_2/h, m_3/h^2, \dots$ also defines an $L_\infty$~structure of $\Pi A[[h]]$.
\end{proposition}

\begin{proof}
We first need to check that each of the maps $m_n/h^n$ do indeed give maps $A[[h]]\to A[[h]]$ (as opposed to just maps $A((h))\to A((h))$) and secondly that they do indeed give an $L_\infty$~structure.

Since each $D_i$ is an $(i+1)$--th order differential operator the operation $m_n$ is given by a formula involving only the operators $h^{i-1}D_{i-1}$ for $i \geq n$. But this means that the expression for $m_n$ has a factor of $h^{n-1}$ and so $m_n/h^{n-1}$ is a map $A[[h]]\to A[[h]]$.

To see that this does indeed define an $L_\infty$~structure observe that it is obtained by conjugating the derivation $m= m_1^* + m_2^* + \dots$ with the automorphism $\widehat{S}A((h))^* \to \widehat{S}A((h))^*$ given by setting $a \mapsto h a$ on $A$, extending $h$--linearly and then extending to an automorphism of $\widehat{S}A((h))^*$.
\end{proof}

\begin{remark}\label{rem:genericfibre}
The two $L_\infty$~structures on $\Pi A[[h]]$ in \autoref{prop:linftyrescale} of course give rise to $L_\infty$~structures on $\Pi A((h))$ and the proof of \autoref{prop:linftyrescale} shows that these $L_\infty$~structures are isomorphic. However, it does not follow that this is necessarily the case for the two $L_\infty$~structures on $\Pi A[[h]]$. More precisely, we can view these $L_\infty$~structures as formal deformations of two different $L_\infty$~structures on $\Pi A$. The first deformation is always trivial, by \autoref{thm:homotopyabelian}, whereas the second is trivial, as we shall see, in the presence of the degeneration condition.
\end{remark}

\begin{definition}\label{def:linfbv}
Let $A$ be a $\bvinf$~algebra. We denote by $\mathfrak{g}[[h]] = \Pi A[[h]]$ the $L_\infty$~algebra with $L_\infty$ structure given by $m_1 = D, m_2/h, m_3/h^2, \dots$. We denote by just $\mathfrak{g} = \Pi A$ the $L_\infty$~algebra obtained by setting $h=0$. In particular the differential is then just $m_1 = D_0$.
\end{definition}

The $L_\infty$~structure on $\mathfrak{g} = \Pi A$ is given explicitly by the following formulae for the maps $m_n\co A\to A$.
\[
m_n(a_1,\dots, a_n) = [[[\dots [D_{n-1},a_1]\dots],a_{n-1}],a_n](1)
\]

\begin{remark}
The above explicit formulae imply that the $L_\infty$~operations $m_n$ are derivations in each variable (or multiderivations).
\end{remark}

The following theorem, which is the central result of this section, will be our main tool. It says that $L_\infty$~algebras arising in this way from $\bvinf$~algebras with the degeneration property are homotopy abelian. This is an $L_\infty$~generalisation of the theorem proved in \cite{katzarkovkontsevichpantev2008:hodgetheoretic,terilla2008:smoothnessbv}, which says that dg Lie algebras arising from dg BV algebras with the degeneration property are homotopy abelian.

\begin{theorem}
Let $A$ be a $\bvinf$~algebra with the degeneration property. Then the associated $L_\infty$~algebra $\mathfrak{g}=\Pi A$ is homotopy abelian.
\end{theorem}

\begin{proof}
First note that for any $\bvinf$~algebra it is the case that $\mathfrak{g}((h)) = \mathfrak{g}[[h]][h^{-1}]$ is homotopy abelian, since by \autoref{rem:genericfibre} $\mathfrak{g}((h))$ is isomorphic to the $L_\infty$~structure on $\Pi A((h))$ associated to the operator $D$, regarded as an operator on $A((h))$. But by \autoref{thm:homotopyabelian} this is homotopy abelian.

The degeneration property means that there exists a $\ground[[h]]$--linear deformation retract of the chain complex $\mathfrak{g}[[h]]$ onto its homology. Therefore the $L_\infty$~operations of the minimal model of $\mathfrak{g}[[h]]$ are also $\ground[[h]]$--linear and tensoring with $\ground((h))$ over $\ground[[h]]$ gives the $L_\infty$~minimal model of $\mathfrak{g}((h))$. However, since $\mathfrak{g}((h))$ is homotopy abelian, all the $L_\infty$~operations of this minimal model are zero, and so the same is true for the operations of the minimal model of $\mathfrak{g}[[h]]$ and hence, setting $h=0$, also for the operations of the minimal model of $\mathfrak{g}$. Therefore $\mathfrak{g}$ is homotopy abelian.
\end{proof}

\section{\texorpdfstring{$\bvinf$}{BV-infinity}~structure on the de Rham algebra}
Let $M$ be a (super)manifold of dimension $d$ with space of functions $C^\infty(M)$. Denote by ${\mathfrak A}(M)$ the space of global sections of the super vector bundle $\bigwedge^\bullet TM$, in other words ${\mathfrak A}(M) = \bigwedge^\bullet_{C^\infty(M)}\Der(C^{\infty}(M))$. This is the space of multivector fields on $M$.

Recall that $\mathfrak{A}(M)$ has the structure of a Gerstenhaber algebra, with the commutative product given by the wedge product and the antibracket given by the Schouten bracket, which is the unique way of extending the Lie derivative to make $\mathfrak{A}(M)$ into a Gerstenhaber algebra. More precisely in terms of the Lie bracket on vector fields we have
\[
[ v_1\wedge\dots \wedge v_n, w_1\wedge\dots \wedge v_m ] = \sum_{i,j}(-1)^{i+j}[v_i,w_j]\wedge v_1\wedge\dots\wedge \hat v_i \wedge \dots\wedge v_n \wedge w_1 \wedge \dots \wedge \hat w_j \wedge \dots \wedge v_m
\]
for vector fields $v_i$, $w_j$ and for a function $f$ and vector field $v$
\[
[v,f] = v(f).
\]
Recall that a Poisson structure on $M$ is a bivector field on $M$ whose Schouten square is zero. Considering general multivector fields, we obtain the notion of a generalised (or higher) Poisson structure, cf.~\cite{voronovkhudaverdian2008:higerpoissonbrackets}.

\begin{definition}
A generalised Poisson structure on $M$ is an even element $P\in\mathfrak A(M)$ for which $[P,P]=0$.
\end{definition}

Any generalised Poisson structure $P$ can be written as $P=P_{-1}+P_0+\dots$ where $P_i\in\bigwedge^{i+1}_{C^\infty(M)}\Der(C^{\infty}(M))$. For simplicity (i.e.~in order to exclude considering curved $L_\infty$~algebras) we make the blanket assumption that $P_{-1}=0$. Furthermore, if $M$ is a purely even manifold of dimension $d$, then $P_0$ must likewise be zero and $P=P_1+\dots+P_d$.

For a generalised Poisson structure $P=\sum_nP_n$ as above let $P(h)\in \mathfrak{A}(M)[[h]]$ be defined by the formula:
\[P(h)=\sum_nP_nh^n\]
Then clearly we have $[P(h),P(h)]=0$ in $\mathfrak{A}(M)[[h]]$.

Note that the $n$--vector field $P_n$ acts by the Lie derivative on the de Rham algebra $\Omega(M)$. Recall that the operator of the Lie derivative along a multivector field $Q$ is defined as $L_Q=[i_Q,d]$ where $i_Q$ is the operation of the interior derivative and $d$ is the de Rham differential. Since for two multivector fields $Q_1$ and $Q_2$ we have $i_{Q_1\wedge Q_2}=\pm i_{Q_1}\circ i_{Q_2}$ and since $d$ is a derivation of $\Omega(M)$, we conclude that $L_{P_n}$ is a differential operator of order $n$ on $\Omega(M)$. Let $L_{P}(h)=L_{P(h)}$; it is an operator on $\Omega(M)[[h]]$. The identity $[P(h),P(h)]=0$ implies $L_P(h)\circ L_P(h)=0$.

\begin{definition}
The sequence of operators $L_{P_n}$ on $\Omega(M)$ determines the structure of a $\bvinf$~algebra structure on $\Omega(M)$, which will be referred to as the de Rham--Koszul $\bvinf$~algebra of $M$. The $L_\infty$~algebra on $\Omega(M)$ associated to this $\bvinf$~structure according to \autoref{def:linfbv} will be called the de Rham--Koszul $L_\infty$~algebra of $M$.
\end{definition}

\begin{remark}
The de Rham--Koszul $L_\infty$~structure on $\Omega(M)$ defined above was introduced in \cite{voronovkhudaverdian2008:higerpoissonbrackets} and the corresponding $L_\infty$~operations were called `higher Koszul brackets' there. Recall that $\Omega(M)[[h]]$ possesses another $L_\infty$~structure, which is a trivial deformation of the homotopy abelian structure on $\Omega(M)$, cf.~\autoref{rem:genericfibre}. The latter structure was  considered in \cite{bruce:higherpoission}.
\end{remark}

\begin{theorem}\label{thm:generalisedPoisson}
Let $M$ be a manifold with a generalised Poisson structure; then the $\bvinf$~algebra $\Omega(M)$ satisfies the degeneration property.
\end{theorem}

\begin{proof}
For simplicity we assume that $M$ is a purely even manifold, although the arguments carry over with obvious modifications to the supergraded case. Let $\Omega$ be the graded sheaf of differential forms on $M$; thus for an open set $U\in M$ the group $\Omega(U)$ is the de Rham algebra on $U$.  For a generalised Poisson structure $P= P_1 + P_2 + \dots$ on $M$ the operators $L_{P_i}$ are determined locally and can be viewed as endomorphisms of the sheaf $\Omega$. Adjoining the formal variable $h$ we can, therefore, consider a dg sheaf $\Omega[[h]]$ with the differential $d + h P_1+\dots$ as well as its truncated version $\Omega[h]/h^N$ for $N=1,2,\dots$; here $d$ is the de Rham differential. Since the sheaf $\Omega$ is fine, the degeneration property for the $\bvinf$~algebra $\Omega(M)$ is equivalent to the statement that the hypercohomology of $\Omega[h]/h^N$ is a free ${\mathbb R}[h]/h^N$--module for any $N=1,2,\dots$.

Let $U$ be a contractible open set in $M$. Then, filtering the dg space $(\Omega(U)[h]/h^N, d+hL_{P_1}+\dots+h^{N-1}L_{P_{N-1}})$ by the powers of $h$ and using the fact that the de Rham differential $d$ is acyclic, we conclude that it is quasi-isomorphic to ${\mathbb R}[h]/h^N$ concentrated in degree zero. It follows that the dg sheaf $\Omega$ is quasi-isomorphic to the constant sheaf ${\mathbb R}[h]/h^N$. Therefore its hypercohomology is isomorphic to $H(M)[h]/h^N$ and is free over ${\mathbb R}[h]/h^N$.
\end{proof}

In the case of an ordinary Poisson structure a statement equivalent to \autoref{thm:generalisedPoisson} was proved in~\cite{dotsenkoshadrinvallette:derhamhomotopyfrob}. The following corollary generalises the corresponding results of \cite{sharygintalalaev2008:formalitypoisson} and \cite{fiorenzamanetti2012:koszulbrackets} formulated for ordinary Poisson manifolds.

\begin{corollary}
The de Rham--Koszul $L_\infty$~algebra of $M$ is homotopy abelian. In particular, the higher Koszul brackets on $H(M)$ vanish.\qed
\end{corollary}

Note that $\Omega(M)$ can be viewed as a double complex with $\Omega^{p,q}=\Omega^{p-q}$ and with two commuting differentials $d$ and $L_P$. Then we have the following corollary which was obtained in the case of ordinary Poisson manifold in \cite{fernandezetal1998:canonicalspectralsequence}.

\begin{corollary}
The spectral sequence of the double complex $\Omega(M)$ collapses at the $E_1$ term $E_1=H(\Omega(M),d)$.\qed
\end{corollary}

\bibliography{references}
\bibliographystyle{alphaurl}

\end{document}